\documentclass{article}

\usepackage{enumerate}
\usepackage{amsmath}

\usepackage{float}

\usepackage{amssymb}
\usepackage{amsfonts}
\usepackage{amssymb}
\usepackage{amsthm}
\usepackage{newlfont}
\usepackage{mathtools}
\usepackage[top=1in, bottom=1in, left=1in, right=1in]{geometry}
\usepackage{tikz}

\theoremstyle{plain}
\newtheorem{theorem}{Theorem}[section]
\newtheorem{proposition}[theorem]{Proposition}
\newtheorem{corollary}[theorem]{Corollary}
\newtheorem{lemma}[theorem]{Lemma}

\theoremstyle{definition}

\newtheorem{example}{Example}[section]
\newtheorem*{remark}{Remark}

\newcommand{\R}{\mathbb{R}}
\newcommand{\N}{\mathbb{N}}
\newcommand{\Z}{\mathbb{Z}}

\newcommand{\Deg}{\operatorname{Deg}}

\newcommand{\eps}{\varepsilon}

\newcommand{\Hidden}[1]{}

\begin{document}

\title{Liouville property and non-negative Ollivier curvature on graphs}
\author{J\"urgen Jost\footnote{MPI Leipzig, jost@mis.mpg.de}~~~~~
~~~~~Florentin M\"unch\footnote{MPI Leipzig, muench@mis.mpg.de}~~~~~Christian Rose\footnote{MPI Leipzig, crose@mis.mpg.de}}
\date{\today}
\maketitle

\begin{abstract}
For graphs with non-negative Ollivier curvature, we prove the Liouville property, i.e., every bounded harmonic function is constant. 
As applications, we apply this result to zero range processes on the line and to lattices with potentials satisfying some convexity condition.
Moreover, we improve Ollivier's results on concentration of the measure under positive Ollivier curvature.
\end{abstract}

%%%%%%%%%%%%%%
%\section{Introduction}
%%%%%%%%%%%%%%

\section{Introduction}
Generally, it seems to be very hard to derive analytic or geometric properties from non-negative Ollivier curvature. Indeed, not many results of this kind seem to be known yet.
We prove that graphs with non-negative Ollivier curvature satisfy the Liouville property which seems to be the first analytic result under the assumption of non-negative Ollivier curvature.

In contrast, non-negative Bakry Emery curvature has strong, well known implications on the heat semigroup. In particular, the gradient of a  bounded solution to the heat equation decays like $1/\sqrt{t}$ or faster \cite{lin2015equivalent,gong2015properties,
keller2018gradient}
 which implies Harnack \cite{chung2014harnack} and Buser inequalities \cite{liu2015curvature, 
klartag2015discrete,liu2018eigenvalue,
liu2018buser}, lower diameter bounds in terms of the spectral gap \cite{chung2014harnack}, and the Liouville property \cite{hua2017liouville} which can be proven almost immediately using the gradient decay.
Using a non-linear modification of the Bakry Emery curvature, on can derive even stronger Li-Yau type gradient estimates \cite{munch2014li,bauer2015li,dier2017discrete,
horn2014volume,munch2017remarks}.

To establish this gradient decay under non-negative Ollivier curvature is one of the major open problems in this subject.
Therefore it is an important step in the study of Ollivier curvature to investigate the Liouville property which is closely related to the gradient decay. 

\begin{figure}[H]
\centering
  \includegraphics[width=12cm]{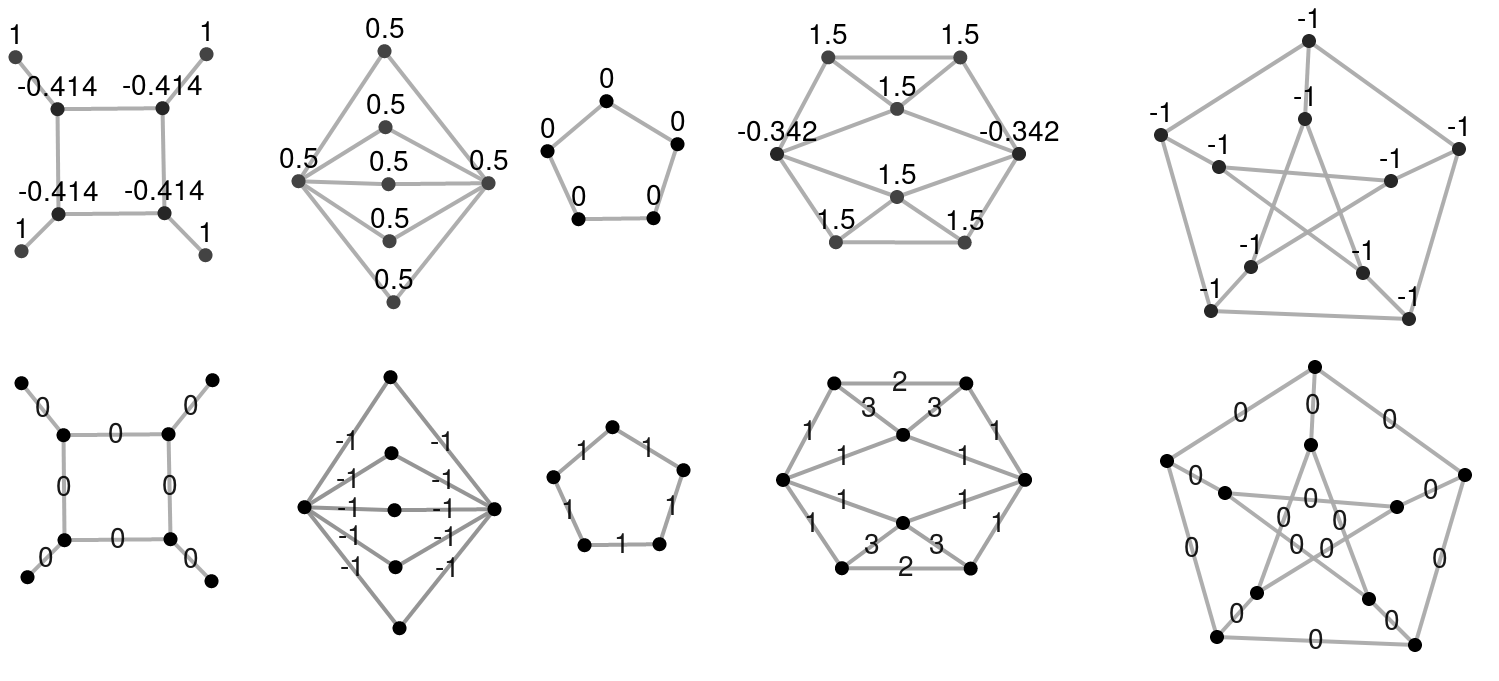}
  \caption{The upper row shows the non-normalized Bakry-Emery curvature. The lower row shows the non-normalized Ollivier curvature from \cite{munch2017ollivier}. The curvature was calculated by the graph curvature calculator \cite{cushing2017graph}.}
  \label{fig:Curv}
\end{figure}

As demonstrated in Figure~\ref{fig:Curv}, there is no implication between non-negative Bakry-Emery and non-negative Ollivier curvature.
We remark that Bakry-Emery and Ollivier curvature are intrinsically different in the sense that Ollivier curvature is defined on edges and Bakry Emery curvature is defined on vertices and can be considered as an analog of the minimal eigenvalue of the Ricci curvature tensor in a point.
Moreover in this paper, we refer to Lin-Lu-Yau's modification of Ollivier curvature which corresponds to lazy random walks and is always larger or equal to Ollivier curvature for non-lazy random walks, see \cite{lin2011ricci}.

Another important question we address in this article is to give interesting examples of infinite graphs with non-negative Ollivier curvature.
Previously known infinite graphs with non-negative Ollivier curvature are certain birth death chains, antitrees \cite{cushing2020curvature} and Cayley graphs for which the Liouville property is either well known, or our theorem is not applicable.
Here, we show that zero range processes on the line, and lattices with suitable potentials also have non-negative Ollivier curvature and therefore satisfy the Liouville property.

Although there many open questions for non-negative Ollivier curvature, the case of positive Ollivier curvature is well understood.
In particular, a positive lower bound on the Ollivier curvature implies an upper diameter bound, eigenvalue estimates, and concentration of the measure \cite{ollivier2009ricci, bauer2011ollivier}.
In this note, we improve the concentration of the measure by applying the methods from \cite{schmuckenschlager1998curvature}.

%Bobo Hua proved the Liouville property for graphs with non-negative Bakry Emery curvature \cite{hua2017liouville}. 

% In contrast, positive Ollivier curvature bounds are well investigated and

\subsection{Setup and notation}
A measured and weighted graph $G=(V,w,m)$ is triple consisting of a countable set $V$, a symmetric function $w:V\times V \to [0,\infty)$ which is zero on the diagonal, and a function $m:V\to (0,\infty)$.
We write $x\sim y$ whenever $w(x,y)>0$.
We will always assume local finiteness, i.e., for all $x \in V$,
\[
|\{y: w(x,y)>0\}|<\infty.
\]
We write $q(x,y):=w(x,y)/m(x)$ and define $\Delta: \R^V \to \R^V$ by
\[ \Delta f(x) := \sum_y q(x,y)(f(y)-f(x)).\]
Note that $\Delta \leq 0$, i.e., $\sum_{x} m(x)f(x)\Delta f(x) \leq 0$ for all finitely supported $f:V \to \R$.
We say a function $f \in \R^V$ is \emph{harmonic} if $\Delta f = 0$.
We denote the \emph{weighted vertex degree} of $x\in V$  by $\Deg(x):=\sum_y q(x,y)$, see e.g. \cite[Section~2.2]{huang2013note}. In the Markov chain setting, the weighted vertex degree is usually called jump rate $J(x)$, see e.g. \cite{fathi2015curvature}.
We write $$\Deg_{\max}:=\sup_{x\in V} \Deg(x) \qquad \mbox{ and } \qquad q_{\min}:=\inf_{x\sim y} q(x,y).$$
The \emph{combinatorial graph distance} is given by
\[
d(x,y) := \inf\{n:x=x_0 \sim \ldots \sim x_n=y \}.
\]
Given the graph distance, we define the gradient
$\nabla_{xy} f$ for $f \in \R^V$ and $x \neq y \in V$ via
\[
\nabla_{xy} f := \frac{f(x)-f(y)}{d(x,y)}.
\]
for $f \in \R^V$, we write $\|f\|_\infty := \sup_{x\in V} |f(x)|$ and 
\[
\|\nabla f\|_\infty := \sup_{x\sim y} \nabla_{xy} f.
\]
The Ollivier curvature, also called coarse Ricci curvature, was introduced in \cite{ollivier2007ricci,ollivier2009ricci} for discrete Markov chains.
Modifications have been given defined in \cite{lin2011ricci} and \cite{jost2014ollivier} in order to compute the curvature of random graphs and to relate curvature to the clustering coefficient.
In this article, we use the generalized version of Ollivier curvature from \cite{munch2017ollivier} which is applicable to all weighted graph Laplacians.
By \cite{munch2017ollivier}, the Ollivier curvature $\kappa(x,y)$ for $x\neq y \in V$ is given by
\[
\kappa(x,y) := \inf_{\substack{\nabla_{yx}f=1 \\ \|\nabla f\|_\infty = 1}} \nabla_{xy} \Delta f. 
\]
This definition coincides with the modified curvature introduced by Lin, Lu, Yau \cite{lin2011ricci} whenever the latter is defined, i.e., whenever 
$\Deg \equiv 1$ and $w(x,y) \in \{0,1\}$, see \cite[Theorem~2.1]{munch2017ollivier}.
By \cite[Proposition~2.4]{munch2017ollivier}, the curvature can also be calculated via transport plans.
Connecting the Lipschitz functions to optimal transport plans is a crucial step for proving the Liouville property. Therefore, we recall \cite[Proposition~2.4]{munch2017ollivier}.

\begin{proposition}[{{See \cite[Proposition~2.4]{munch2017ollivier}}}]\label{pro:CharTransport}
Let $G=(V,w,m)$ be a graph and let $x_0 \neq y_0$ be vertices. Then,
\begin{align}
\kappa(x_0,y_0) &= \sup_{\rho}  \sum_{\substack{x \in B_1(x_0) \\ y \in B_1(y_0)}}\rho(x,y) \left[1 - \frac{d(x,y)}{d(x_0,y_0)}\right] \label{eq:PropTransport} 
\end{align}
where the supremum is taken over all $\rho: B_1(x_0) \times B_1(y_0) \to [0,\infty)$ such that
\begin{align}
\sum_{y \in B_1(y_0)} \rho(x,y) &= q(x_0,x)  \qquad \mbox{ for all } x \in S_1(x_0) \mbox{ and} \label{eq:rhoXProp}\\
\sum_{x \in B_1(x_0)} \rho(x,y) &= q(y_0,y) \qquad \mbox{ for all } y \in S_1(y_0)  \label{eq:rhoYProp}.
\end{align}
\end{proposition}
We remark that $\rho$ is defined on balls, but we only require the coupling property on spheres.  Moreover, we do no not assume that $\rho$ is a probability measure.
A function $\rho$ attaining the supremum in (\ref{eq:PropTransport}) is called \emph{optimal transport plan}. Due to compactness, there always exists an optimal transport plan.

\section{Liouville property and non-negative Ollivier curvature}

The study of harmonic functions and, in particular, the Liouville property on manifolds with non-negative Ricci curvature traces back to \cite{yau1975harmonic} and is still matter of current research \cite{colding2019liouville}.
Liouville type properties on graphs have been studied in e.g. \cite{woess2000random,masamune2009liouville,
benjamini1996harmonic}.
We now present our main theorem stating that every bounded harmonic function is constant when assuming non-negative Ollivier curvature.

%We write $q(x,y):=w(x,y)/m(x)$
\begin{theorem}\label{thm:Liouville}
Let $G=(V,w,m)$ be a graph. Suppose
\begin{itemize}
\item $\Deg_{\max} < \infty$,
\item
$q_{\min}>0$,
\item $\kappa(x,y) \geq 0$ for all $x \neq y \in V$.
\end{itemize}
Then, every bounded harmonic function is constant.
\end{theorem}

We remark that the assumption $\kappa(x,y)\geq 0$ is weaker than assuming non-negative Ollivier curvature in the non-lazy random walk setting. 
In order to prove the theorem, we first need a lemma concerning transport plans, stating that if $\kappa(x_0,y_0)\leq \eps$, then there exists an optimal transport plan which transports a significant amount of mass over the distance $d(x_0,y_0)+1$. 

\begin{lemma}\label{lem:transport}
Suppose $d(x_0,y_0)\kappa(x_0,y_0) \leq \eps$ for some $x_0,y_0 \in V$ and some $\eps>0$. 
Then, there exists an optimal transport plan $\rho:B_1(x_0) \times B_1(y_0) \to [0,\infty)$ s.t.
\[
\sum_{\substack{x \in B_1(x_0) \\y \in B_1(y_0)\\d(x,y) > d(x_0,y_0)}} \rho(x,y) \geq (q_{\min} - \eps)/2.
\]
\end{lemma}
Remarkably, this lemma fails in the non-lazy random walk setting as one can see on the one-dimensional lattice $\Z$ with standard weights.

\begin{proof}
Let $\rho_0$ be an optimal transport plan and
let $x' \sim x_0$  s.t. $d(x',y_0)=d(x_0,y_0)-1$.
We want to construct an optimal transport plan transporting a significant mass over the distance $d(x_0,y_0) + 1$. To this end, we construct an optimal transport plan transporting a significant amount of mass over a distance shorter than $d(x_0,y_0)$ which will be useful since the average transport distance is close to $d(x_0,y_0)$ if the curvature is small. In particular, our transport plan will have the property that $x'$ is transported only to vertices $y \in B_1(y_0)$ with $d(x',y) \leq d(x_0,y_0)-1$.
We define a map $\rho:B_1(x_0) \times B_1(y_0) \to [0,\infty)$ which shall be our new optimal transport plan via
\[
\rho(x,y) := \begin{cases}
\rho_0(x,y_0) + \sum_{\substack{z \in B_1(y_0)\\d(x',z) \geq d(x_0,y_0)}} \rho_0(x',z)
&: x=x',y=y_0,\\
0&: x=x', d(x',y)\geq d(x_0,y_0) ,\\
\rho_0(x_0,y) + \rho_0(x',y) &: x=x_0, d(x',y) \geq d(x_0,y_0),\\
\rho_0(x,y)&: \mbox{otherwise}.
\end{cases}
\]
We now prove that $\rho$ is also an optimal transport plan.
To this end, we first show that $\rho$ satisfies the marginal conditions.
For $x=x'$, we have
\begin{align*}
\sum_{y\in B_1(y_0)} \rho(x',y) &= 
\rho(x',y_0) + \sum_{d(x',y) \geq d(x_0,y_0)} \rho (x',y) + \sum_{\substack{d(x',y) < d(x_0,y_0)\\y\neq y_0}} \rho(x',y)\\
&=\left( \rho_0(x',y_0) + \sum_{d(x',y) \geq d(x_0,y_0)} \rho_0(x',y) \right) + 0 + \sum_{\substack{d(x',y) < d(x_0,y_0)\\y\neq y_0}} \rho_0(x',y)\\
&=\sum_{y\in B_1(y_0)} \rho_0(x',y)\\
&=q(x_0,x').
\end{align*}
For $x \in S_1(x_0)\setminus \{x'\}$, we have
$\rho(x,y) = \rho_0(x,y)$ for all $y \in B_1(y_0)$, and thus,
\[
\sum_{y\in B_1(y_0)} \rho(x,y) = \sum_{y\in B_1(y_0)} \rho_0(x,y) = q(x_0,x).
\]
For $y \in S_1(y_0)$ s.t. $d(x',y) < d(x_0,y_0)$, we have $\rho(x,y) = \rho_0(x,y)$ for all $x \in B_1(x_0)$, and thus,
\[
\sum_{x\in B_1(x_0)} \rho(x,y) = \sum_{x\in B_1(x_0)} \rho_0(x,y) = q(y_0,y).
\]
For $y \in S_1(y_0)$ s.t. $d(x',y) \geq d(x_0,y_0)$, we have
\begin{align*}
\sum_{x\in B_1(x_0)}\rho(x,y) &= \rho(x_0,y)+ \rho(x',y) + \sum_{x \in S_1(x_0)\setminus \{x'\}} \rho(x,y) \\
&=\Big(\rho_0(x_0,y) + \rho_0(x',y) \Big) + 0 + \sum_{x \in S_1(x_0)\setminus \{x'\}} \rho_0(x,y) \\
&=\sum_{x\in B_1(x_0)}\rho_0(x,y) = q(y_0,y).
\end{align*}
This proves that $\rho$ is indeed a transport plan.
In order to show that $\rho$ is optimal, it is sufficient by optimality of $\rho_0$ to show
\[
\sum_{\substack{x\in B_1(x_0)\\y \in B_1(y_0)}} \rho(x,y)\Big(d(x_0,y_0) - d(x,y) \Big) \geq \sum_{\substack{x\in B_1(x_0)\\y \in B_1(y_0)}} \rho_0(x,y)\Big(d(x_0,y_0) - d(x,y) \Big)
\]
We write $C(x,y):= \Big(\rho(x,y)-\rho_0(x,y)\Big)\Big(d(x_0,y_0) - d(x,y) \Big)$.
Thus, it suffices to prove
\[
\sum_{\substack{x\in B_1(x_0)\\y \in B_1(y_0)}} C(x,y) \geq 0.
\]
Since $C(x,y)=0$ whenever $\rho(x,y)=\rho_0(x,y)$, we have
\begin{align*}
\sum_{\substack{x\in B_1(x_0)\\y \in B_1(y_0)}} C(x,y) = C(x',y_0) + \sum_{d(x',y) \geq d(x_0,y_0)} \Big( C(x',y) + C(x_0,y) \Big).
\end{align*}
Observe that since $\rho(x',y)=0$ whenever $d(x',y) \geq d(x_0,y_0)$, we have
\[
C(x',y_0) = \left(\sum_{d(x',y) \geq d(x_0,y_0)} \rho_0(x',y)\right) \cdot \Big(d(x_0,y_0)-d(x',y_0) \Big) = \sum_{d(x',y) \geq d(x_0,y_0)} \rho_0(x',y).
\]
If $d(y,x') \geq d(x_0,y_0)$, we have
\[
C(x',y) = -\rho_0(x',y) \Big(d(x_0,y_0)-d(x',y) \Big)
\]
and
\[
C(x_0,y)= \rho_0(x',y) \Big(d(x_0,y_0) - d(x_0,y)\Big)
\]
and thus,
\[
C(x',y) + C(x_0,y) = \rho_0(x',y)\Big(d(x',y) - d(x_0,y)\Big) \geq - \rho_0(x',y)
\]
since $|d(x',y) - d(x_0,y)| \leq d(x_0,x')=1$.
Putting together gives
\begin{align*}
\sum_{\substack{x\in B_1(x_0)\\y \in B_1(y_0)}} C(x,y) &= C(x',y_0) + \sum_{d(x',y) \geq d(x_0,y_0)} \Big( C(x',y) + C(x_0,y) \Big) \\
&\geq \sum_{d(x',y) \geq d(x_0,y_0)} \rho_0(x',y) - \sum_{d(x',y) \geq d(x_0,y_0)} \rho_0(x',y) = 0
\end{align*}
which proves that $\rho$ is an optimal transport plan.
Observe that via the transport plan $\rho$ the vertex
$x'$ is transported only to vertices $y$ with $d(x',y) < d(x_0,y_0)$,
i.e.,
\[
q(x_0,x')=\sum_y \rho(x',y) = \sum_{d(x',y)<d(x_0,y_0)} \rho(x',y).
\]
Thus, we have
\begin{align*}
\eps \geq d(x_0,y_0)\kappa(x_0,y_0) &=
\sum_{\substack{x \in B_1(x_0) \\y \in B_1(y_0)}} \rho(x,y)(d(x_0,y_0) - d(x,y))
\\&\geq  \sum_{\substack{x \in B_1(x_0) \\y \in B_1(y_0)\\d(x,y) < d(x_0,y_0)}} \rho(x,y) - 2\cdot \sum_{\substack{x \in B_1(x_0) \\y \in B_1(y_0)\\d(x,y) > d(x_0,y_0)}} \rho(x,y) \\
&\geq  \sum_{\substack{y \in B_1(y_0)\\d(x,y) < d(x_0,y_0)}} \rho(x',y) - 2\cdot \sum_{\substack{x \in B_1(x_0) \\y \in B_1(y_0)\\d(x,y) > d(x_0,y_0)}} \rho(x,y) \\
&= q(x_0,x') - 2\cdot \sum_{\substack{x \in B_1(x_0) \\y \in B_1(y_0)\\d(x,y) > d(x_0,y_0)}} \rho(x,y)
\end{align*}
where the second estimate follows from $d(x_0,y_0)-d(x,y) \in \{-2,-1,0,1,2\}$.

Hence,
\[ \sum_{\substack{x \in B_1(x_0) \\y \in B_1(y_0)\\d(x,y) > d(x_0,y_0)}} \rho(x,y)
\geq (q(x_0,x') - \eps)/2 \geq (q_{\min}-\eps)/2.
\]
This finishes the proof.
\end{proof}

For simplicity, we write $D:=\Deg_{\max}$ and $q:= q_{\min}$.
\begin{lemma}\label{lem:inductive}
Let $f$ be a harmonic function with $\|\nabla f\|_\infty=1$.
Let $\eps \in (0,q/4D)$ and let $x_0\neq y_0$ s.t. $f(x_0)-f(y_0)\geq  d(x_0,y_0) - \eps$.
Then, there exists $x' \in B_1(x_0)$ and $y' \in B_1(y_0)$ s.t.
%\begin{enumerate}[(i)]
\begin{itemize}
\item %$d(x',y')\nabla f(x',y')
$f(x')-f(y')> d(x',y') - \eps \cdot 10 D/q$,
\item $d(x',y')>d(x_0,y_0)$.
\end{itemize}
%\end{enumerate}
\end{lemma}

\begin{proof}

Define
$g_0:B_1(x_0) \to \R$ via
\[
g_0(w):= f(w) \wedge (f(x_0) - d(x_0,y_0) + d(y_0, w)).
\]
Then, $g_0$ is $1$-Lipschitz as the minimum of two $1$-Lipschitz functions.
Let $g:V \to \R$ the minimal Lipschitz extension of $g_0$ given by
\[
g(z):= \max_{w \in B_1(x_0)} g_0(w)-d(w,z).
\]
For all $z \in V$ and all $w \in B_1(x_0)$, the condition $\|\nabla f \|_\infty=1$ yields
\[
f(z) \geq f(w)-d(z,w) \geq g_0(w) -d(z,w)
\]
which implies $f \geq g$.
Since $g_0$ is $1$-Lipschitz, we have $g_0=g|_{B_1(x_0)}$.
Observe that $g$ is $1$-Lipschitz as a maximum of $1$-Lipschitz functions.
Thus,
$g(y_0) \geq g(x_0)-d(x_0,y_0)$.
On the other hand, we have
\[
g(y_0) \leq \max_{w \in B_1(x_0)} \Big(f(x_0) - d(x_0,y_0) + d(y_0,w)\Big) - d(w,y_0) = f(x_0)-d(x_0,y_0).
\]
Since $f(x_0)=g_0(x_0)=g(x_0)$, we obtain
$g(y_0) \leq f(x_0)-d(x_0,y_0) = g(x_0) - d(x_0,y_0) \leq g(y_0)$ implying $g(y_0) = g(x_0) - d(x_0,y_0)$.
%Then, $g\leq f$ and $g(x_0)=f(x_0)$ and $g(y_0) = g(x_0) - d(x_0,y_0)$.
Moreover for $w \in B_1(x_0)$, 
\[
f(x_0) - d(x_0,y_0) + d(y_0,w) \geq f(y_0) - \eps + d(y_0,w) \geq f(w) - \eps
\]
yielding $g(w)=g_0(w) \geq f(w)-\eps$.
%$g|_{B_1(x_0)} \geq f|_{B_1(x_0)} - \eps$.
Thus,
$\Delta g(x_0) \geq \Delta f(x_0) - D\eps$ by the definition of $\Delta$.
Since $g(y_0) = f(x_0) - d(x_0,y_0) \geq f(y_0) - \eps$ and since $g \leq f$, we have
$\Delta g(y_0) \leq \Delta f(y_0) + D\eps$.
Since $\nabla_{x_0y_0} g=\|\nabla g\|_\infty =1$, the definition of $\kappa$ gives
\[
d(x_0,y_0)\kappa(x_0,y_0) \leq \Delta g(y_0) - \Delta g(x_0) \leq 2D\eps. 
\]
We have $g(y)-g(x) \geq  - d(x,y)$.
Let $$H:= \min_{\substack{x \in B_1(x_0) \\y \in B_1(y_0)\\d(x,y) > d(x_0,y_0)}}  d(x,y) - g(x) + g(y) \geq 0$$
where the set from which the minimum is taken is not empty due to Lemma~\ref{lem:transport} and finite due to local finiteness.
Let $\rho$ be the optimal transport plan from Lemma~\ref{lem:transport}.
We write
\begin{align*}
2D\eps \geq \Delta g(y_0) - \Delta g(x_0) &\geq \sum_{\substack{x \in B_1(x_0) \\y \in B_1(y_0)}} \rho(x,y)(g(y)-g(y_0))  - \sum_{\substack{x \in B_1(x_0) \\y \in B_1(y_0)}} \rho(x,y)(g(x)-g(x_0))\\
&= \sum_{\substack{x \in B_1(x_0) \\y \in B_1(y_0)}} \rho(x,y)(g(y)-g(x) + d(x_0,y_0))\\
& \geq \sum_{\substack{x \in B_1(x_0) \\y \in B_1(y_0)}} \rho(x,y)(d(x_0,y_0)-d(x,y)) +  \sum_{\substack{x \in B_1(x_0) \\y \in B_1(y_0)\\d(x,y) > d(x_0,y_0)}} \rho(x,y)(d(x,y) - g(x) + g(y)) 
\\&\geq d(x_0,y_0)\kappa(x_0,y_0) + H \sum_{\substack{x \in B_1(x_0) \\y \in B_1(y_0)\\d(x,y) > d(x_0,y_0)}} \rho(x,y)\\
&\geq H(q-2D\eps)/2
\end{align*}
where the last estimate follows from Lemma~\ref{lem:transport}.
Hence, $H\leq \frac {4D\eps}{q-2D\eps}$.
In particular, there exists $x' \in B_1(x_0)$ and $y'\in B_1(y_0)$ with $d(x',y') > d(x_0,y_0)$ s.t.
\[
g(x')-g(y') \geq d(x',y') - \frac {4D\eps}{q-2D\eps} \geq d(x',y') - 8D\eps/q
\]
where the last estimate follows from $\eps < \frac{q}{4D}$.
We now show that $g$ approximates $f$ in order to lower bound $f(x')-f(y')$.
We have $g(x') \leq f(x')$ and

\begin{align*}
-D\eps \leq \Delta g(x_0) \leq \Delta g(y_0) = \Delta g(y_0)- \Delta f(y_0) &=
D(f(y_0) - g(y_0)) - \sum_{y\sim y_0} q(y_0,y)(f(y)-g(y)) 
\\&\leq D\eps - q(y_0,y')(f(y')-g(y')).
\end{align*}
Thus,
\[
f(y') \leq g(y') + 2D\eps/q(y_0,y') \leq g(y') + 2D\eps/q.
\]
Putting together gives
\[
f(x') - f(y') \geq g(x') - g(y') -2D\eps/q \geq d(x',y') - \eps \cdot 10D/q.
\]
This finishes the proof.
\end{proof}

We recall $D=\Deg_{\max}$ and $q= q_{\min}$.
\begin{proof}[Proof of Theorem~\ref{thm:Liouville}]
Let $f$ be a bounded harmonic function.
Then, $\|\nabla f\|_\infty < \infty$. 
If $f$ is not constant, we
can assume $\|\nabla f\|_\infty=1$ without obstruction.
Let $2\|f\|_\infty < N \in \N$.
Let $\eps < \left(\frac{q}{10D} \right)^N$ be small.
Let $x_0\sim y_0$ s.t. $\nabla f(x_0,y_0)>1-\eps$.
We inductively apply Lemma~\ref{lem:inductive} to construct sequences $(x_n)_{n=0}^N$ and $(y_n)_{n=0}^N$ with the following properties:
\begin{itemize}
\item $d(x_n,y_n) \geq n+1$,
\item $f(x_n) - f(y_n) \geq d(x_n,y_n) - \eps \cdot (10D/q)^n$.
\end{itemize} 
In particular given $x_n$ and $y_n$, we 
have $d(x_n,y_n) \geq n+1$ and $f(x_n)-f(y_n) \geq d(x_n,y_n) - \eps \cdot (10D/q)^n$ by the induction hypothesis. We now 
apply Lemma~\ref{lem:inductive} to obtain $x',y'$ with $d(x',y') > d(x_n,y_n)$ and 
\[
f(x')-f(y') \geq d(x',y') - \left( \eps \cdot (10D/q)^n\right) \cdot  10D/Q = d(x',y') - \eps \cdot (10D/q)^{n+1}.
\]
Thus we set $x_{n+1}:=x'$ and $y_{n+1}:=y'$ which satisfy the desired properties.
In particular,
\[
\|f\|_\infty + \|f\|_\infty \geq f(x_N) - f(y_N) > d(x_N,y_N) - 1 \geq N + 1 - 1 > 2\|f\|_\infty.
\]
This is a contradiction, and thus, $f$ is constant. This finishes the proof.
\end{proof}

\section{Examples}

Besides abelian Cayley graphs, not many infinite Markov chains are known to have non-negative Ollivier curvature.
One interesting example are so called antitrees, even having positive curvature everywhere \cite{cushing2020curvature}, but not satisfying the $q_{\min}$ assumption from our Liouville property theorem.
Here, we present two important example classes of infinite Markov chains with non-negative Ollivier curvature, namely zero range processes on the line and lattices with potential.

\subsection{Zero range processes on the line}
Zero range processes are important interacting particle systems and were introduced by Spitzer in \cite{spitzer1970interaction}.

The zero range process on the infinite line is given by the following data:
\begin{itemize}
\item $n$ particles
\item Vertex set $V=\{x \in \N_0^\Z: \|x\|_1 = n\}$
\item Rate function $r: \N_0 \to \N_0$ non-decreasing with $r(0)=0$ and $r(1)>0$
\item Jump rate $q(x,x - e_i + e_j) := r(x_i)$ whenever $|i-j|=1$
\item $q(x,y)=0$ otherwise
\end{itemize}
Here, $e_i \in \N_0^\Z$ denotes the unit vector, and $x_i$ the $i$-th component of vertex $x$. It is well known that the corresponding continuous time Markov chain is reversible, and therefore can be expressed as a weighted graph.

\begin{theorem}
Zero range processes on the line have non-negative Ollivier curvature. Moreover, $q_{\min} \geq r(1)$ and $\Deg_{\max} \leq 2nr(n)$.
\end{theorem}

\begin{proof}
For showing non-negative Ollivier curvature, it is sufficient to provide a suitable transprt plan.
Let $x \sim y$ and $z:= x\wedge y$.
Then, $x = z + e_i$ and $y=z+e_j$ for some $i,j \in \Z$ with $|i-j|=1$. We assume without obstruction that $j=i+1$
We now define a transport plan $\rho$ for Proposition~\ref{pro:CharTransport} by

\[
\rho(x',y'):= 
  \begin{cases}
    r(x_i)&: x'=y'=y\\
    r(y_j)&: x'=y'=x \\
    r(z_k) &: x'= x -e_k + e_{k\pm 1}, y' = y -e_k + e_{k\pm 1},  \{x',y'\} \cap \{x,y\}=\emptyset\\
    r(x_i)-r(z_i)&: y'=y, x' =x - e_i + e_{i-1}\\
    r(y_j)-r(z_j)&: x'=x, y' = y - e_j +e_{j+1}\\
    r(x_j)&:y'=y,x' = x - e_j + e_i\\
    r(y_i)&:x'=x, y'=y - e_i + e_j\\ 
  \end{cases}
\]
and zero otherwise.
It is straight forward to check that $\rho$ is a feasible transport plan.
Moreover, the first two cases have transport distance zero, the third case transport distance one, and the other cases transport distance two.
Hence,
\[
\sum_{x',y'}\rho(x',y')(d(x,y)-d(x',y')) = r(x_i) + r(y_j) - (r(x_i) - r(z_i)) - (r(y_j) - r(z_j)) - r(x_j) - r(y_i) = 0
\]
as $z_i = y_i$ and $z_j=x_j$.
By Proposition~\ref{pro:CharTransport}, this implies $\kappa(x,y) \geq 0$.
The 'moreover' assertion $q_{\min} \geq r(1)$ is trivial and the claim $\Deg_{\max} \leq 2nr(n)$ follows as every vertex has at most $2n$ neighbors and as the jump rate is at most $r(n)$. This finishes the proof.
\end{proof}
Applying Theorem~\ref{thm:Liouville} yields the following corollary.
\begin{corollary}
Zero range processes satisfy the Liouville property, i.e., every bounded harmonic function is constant.
\end{corollary}

\subsection{Lattices with potential}

On weighted manifolds with weight function $f$, the generalized Ricci curvature satisfies
$\mathrm{Ric} = \mathrm{Ric}_M + \mathrm{Hess}(f)$ \cite{lott2009ricci}.
Particularly, the Ricci curvature on the Euclidean space with convex weight $f$ is non-negative.
Here, we give a discrete analog.
Let $\Z^n$ be the standard lattice and $e^f : \Z^n\to \R_+$.
The weighted jump rates are given by
\[
q(x,y) = 1_{x\sim y} \exp(f(x)-f(y)).
\]
It is easy to show that the corresponding continuous time Markov chain is reversible, and can therefore be expressed as a weighted graph.
The following proposition gives some kind of convexity condition for non-negative Ollivier curvature on the lattice.

\begin{proposition}\label{pro:LatticePotential}
Let $x \in \Z^n$ and  $y=x+e_i$ for some $i \leq n$.
Assume $q$ on the lattice satisfies
\[
q(x,y) + q(y,x) \geq q(x,x-e_i) + q(y,y+e_i) + \sum_{j \neq i} |q(x,x+e_j) - q(y,y+e_j)| + |q(x,x-e_j) - q(y,y-e_j)|.
\]
Then, the edge $(x,y)$ has non-negative Ollivier curvature.
\end{proposition}

\begin{proof}
Choosing the appropriate transport plan, the left hand side is the mass transported by distance zero, and the right hand side is the mass transported by distance two.
Applying Proposition~\ref{pro:CharTransport} shows non-negativity of the Ollivier curvature of edge $(x,y)$.
\end{proof}

We now apply the above proposition to give explicit examples.

\begin{example}
Let $v \in \R^n$ and $x_0 \in V=\Z^n$. Let $c> 0$. Define
\[
f(x) := c\cdot d(x,x_0) + \langle v,x \rangle.
\]
Then, the corresponding Markov chain defined by $q(x,y) = 1_{x\sim y} \exp(f(x)-f(y))$ has non-negative Ollivier curvature. Moreover, all bounded harmonic functions are constant.
\end{example}

\begin{remark}
The case $c=0$ is also allowed, but then the Markov chain is a Cartesian product of biased birth death chains, and non-negativity of the Ollivier curvature is well known in this case.
\end{remark}

\begin{proof}[Proof of the Example]
We first notice that for $i \neq j \leq n$ and all $x \in Z^n$, we have
\[
f(x) - f(x+e_j) = f(x+e_i) - f(x+e_i + e_j).
\] 
Hence, with $y=x+e_i$,
\[
\sum_{j \neq i} |q(x,x+e_j) - q(y,y+e_j)| + |q(x,x-e_j) - q(y,y-e_j)| = 0.
\]
Moreover, $f$ is convex along $x + \Z e_i$ implying
\[
q(x,y) \geq q(y,y+e_i) \qquad \mbox{ and } \qquad q(y,x) \geq q(x,x-e_i)
\]
Combining everything shows non-negative Ollivier curvature by Proposition~\ref{pro:LatticePotential}.
The conditions $q_{\min}>0$ and $\Deg_{max} < \infty$ follow as $f$ has a bounded gradient. Hence, we can apply Theorem~\ref{thm:Liouville} to show that every harmonic function is constant.
This finishes the proof.
\end{proof}

\section{Concentration of measure}
We apply the methods from \cite{schmuckenschlager1998curvature} to improve the concentration of measure results by Ollivier \cite{ollivier2009ricci}.
In \cite{schmuckenschlager1998curvature}, concentration of measure is proved under a positive Bakry Emery curvature bound.
For $$f\in \ell^1(V,m) := \{g \in C(V): \sum_{x \in V} m(x)|g(x)|<\infty\}$$ we write
\[
\langle g \rangle := \sum_{x \in V} m(x)g(x).
\]
We now state our concentration theorem which gives a Gaussian upper bound for the measure of the vertices for which a Lipschitz function deviates from its mean more than $r$.
Non-explicit concentration bounds via transport-information inequalities in terms of Ollivier curvature can also be found in \cite{fathi2015curvature}. 
\begin{theorem}
Let $G=(V,w,m)$ be a graph and let $K>0$. Suppose
\begin{itemize}
\item $\Deg_{\max} \leq 1$,
\item $m(V)=1$,
%\item
%$q_{\min}>0$,
\item $\kappa(x,y) \geq K>0$ for all $x \neq y \in V$.
\end{itemize}
Let $f \in \ell^1(V,m)$ s.t.
\begin{itemize}
\item $\langle f\rangle = 0$,
\item $\|\nabla f\|_\infty \leq 1$.
\end{itemize}
Then,
\[
m(f>r) \leq  e^{-{Kr^2}}.
\]

\end{theorem}
This improves the concentration result by Ollivier \cite[Theorem~33]{ollivier2009ricci} roughly stating
that under the same assumptions as in the theorem,
\[
m(f>r) \leq  e^{-{K^2 r^2}}
\]
if $r$ is not too large.
Firstly, having $K$ in the exponent is better than $K^2$ up to a constant since $K\leq 2$ due to $\Deg_{\max} \leq 1$.
Secondly, our concentration result holds without restricting to small enough $r$.

\begin{proof}
We first observe that $G$ has finite diameter due to  \cite[Proposition~4.14]{munch2017ollivier}.
Thus, $G$ is finite due to local finiteness.
Following \cite{schmuckenschlager1998curvature}, we have
\begin{align*}
-\partial_t \langle e^{\lambda P_t f} \rangle = -
\langle\lambda \Delta P_t f,  e^{\lambda P_t f} \rangle =  \langle \Gamma(\lambda P_t f,e^{\lambda P_t f})\rangle \leq \lambda^2 \langle e^{\lambda P_t f}, \Gamma(P_t f) \rangle
\end{align*}
where $2\Gamma(g,h):= \Delta(fg)-f\Delta g - g \Delta f$.
Moreover by \cite[Theorem~3.8]{munch2017ollivier}, and since $\Deg_{\max} \leq 1$, 
\[
2\Gamma(P_t f) \leq \|\nabla P_t f\|_\infty^2 \leq  e^{-2Kt} \|\nabla f\|_\infty^2 \leq e^{-2Kt}.
\]
Putting together gives
\[
-\partial_t \langle e^{\lambda P_t f} \rangle \leq \frac{\lambda^2 } 2 e^{-2Kt} \langle e^{\lambda P_t f}\rangle.
\]
Integrating from $t=0$ to $\infty$ and applying $\lim_{t\to \infty} e^{\lambda P_t f} = 1$ gives
\[
\langle e^{\lambda  f} \rangle \leq e^{\frac{\lambda^2 }{4K}}.
\]
Thus, we have
\[
m(f>r) \leq e^{-\lambda r}\langle e^{\lambda f} \rangle \leq \exp\left( \frac{\lambda^2 }{4K} - \lambda r \right) = \exp \left(-{Kr^2} \right)
\]
when choosing $\lambda = 2rK$.
This finishes the proof.
\end{proof}

%----------------------------------------------
%\printbibliography

%\bibliographystyle{plain}
%\bibliography{Bibliography}

%\printbibliography
\bibliographystyle{alpha}

%\bibliography{Bibliography.bib}
\newcommand{\etalchar}[1]{$^{#1}$}

\end{document}